\documentclass[11pt,reqno]{amsart}
\usepackage[utf8]{inputenc}
\usepackage{amsmath}
\usepackage{amsfonts}
\usepackage{amssymb}
\usepackage{color}
\usepackage{graphicx}
\usepackage{bbm}
\usepackage{comment}

\swapnumbers

\numberwithin{equation}{section}
\theoremstyle{plain}
\newtheorem{thm}{Theorem}[section]
\newtheorem{proposition}[thm]{Proposition}
\newtheorem{corollary}[thm]{Corollary}
\newtheorem{lemma}[thm]{Lemma}

\theoremstyle{definition}

\newtheorem{remark}[thm]{Remark}

\DeclareMathOperator{\E}{{\mathbb E}}

\DeclareMathOperator{\R}{{\mathbb R}}

\DeclareMathOperator{\N}{{\mathbb N}}

\DeclareMathOperator{\PP}{{\mathbb P}}

 \DeclareMathOperator{\sgn}{sgn}

\newcommand{\EstGen}{\hat F_\Phi}
\newcommand{\EstNorm}{\hat F_p}
\newcommand{\EstAsymp}{\hat{F}_{\Phi,n}}
\newcommand{\Funct}{F_{\Phi}}

\providecommand{\eps}{\varepsilon}
\renewcommand{\phi}{\varphi}

\renewcommand{\subset}{\subseteq}

\providecommand{\abs}[1]{\lvert #1 \rvert}
\providecommand{\norm}[1]{\lVert #1 \rVert}

\renewcommand{\le}{\leqslant}
\renewcommand{\ge}{\geqslant}

\allowdisplaybreaks[4]

\begin{document}
\title[Functional estimation in nonparametric boundary models]{Functional estimation and hypothesis testing in nonparametric boundary models}
\author{Markus Rei\ss}
\author{Martin Wahl}
\address{Markus Rei\ss{}, Institut f\"{u}r Mathematik, Humboldt-Universit\"{a}t zu Berlin, Unter den Linden 6, 10099 Berlin, Germany.}
\email{mreiss@math.hu-berlin.de}
\address{Martin Wahl, Institut f\"{u}r Mathematik, Humboldt-Universit\"{a}t zu Berlin, Unter den Linden 6, 10099 Berlin, Germany.}
\email{martin.wahl@math.hu-berlin.de}

\keywords{Poisson point process, Support estimation, Non-linear functionals, Minimax hypothesis testing.}

\subjclass[2010]{62G05, 62G10, 62G32, 62M30}

\thanks{We are grateful for the helpful comments by the referees. Financial support by the DFG via Research Unit 1735 {\em Structural Inference in Statistics: Adaptation and Efficiency} is acknowledged.}

\begin{abstract}
Consider a Poisson point process  with unknown support boundary curve $g$, which forms a prototype of an irregular statistical model. We address the problem of estimating non-linear functionals of the form $\int \Phi(g(x))\,dx$. Following a nonparametric maximum-likelihood approach, we construct an estimator which is UMVU over H\"older balls and achieves the (local) minimax rate of convergence. These results hold under weak assumptions on $\Phi$ which are satisfied for $\Phi(u)=|u|^p$, $p\ge 1$. As an application, we consider the problem of estimating the $L^p$-norm and derive the minimax separation rates in the corresponding nonparametric hypothesis testing problem. Structural differences to results for regular nonparametric models are discussed.
\end{abstract}

\maketitle
\section{Introduction}

Point processes serve as canonical models for dealing with support estimation. Poisson point processes (PPP) appear in the continuous limit of nonparametric regression models with one-sided or irregular error variables, cf. Meister and Rei\ss\ \cite{Meister13}, and thus form counterparts of the Gaussian white noise (GWN) model. In this paper we consider the observation of a PPP on $[0,1]\times\R$ with intensity function
\begin{equation}\label{EqIntensity}
\lambda_g(x,y)=n\mathbbm{1}(y\ge g(x)),\qquad x\in[0,1],y\in\R,
\end{equation}
where $g$ is an unknown support boundary curve and $n\in\N$. A prototypical regression model corresponding to this PPP model with $n$ replaced by $\alpha n$ is given by
\begin{equation}\label{EqRegression} Y_i=g(i/n)+\eps_i,\quad i=1,\ldots,n,
\end{equation}
with one-sided i.i.d. error variables $\eps_i\ge 0$, satisfying $\PP(\eps_i\le x)= \alpha x+O(x^2)$ as $x\downarrow 0$, cf. the discussion in Rei\ss\ and Selk \cite{ReissSelk14}. The important point to keep in mind is that in these support boundary models, the standard parametric rate is $n^{-1}$ due to the behaviour of extreme value statistics.

In Korostelev and Tsybakov \cite[Chapter 8]{KorTsy93} the problem  of estimating functionals of a binary image boundary from noisy observations has been studied. Although the noise is regular, the Hellinger metric is an $L^1$-distance exactly as in our PPP model. In both models, the minimax rate of convergence for estimating linear functionals of the form $\langle g,\psi\rangle=\int_0^1 g(x)\psi(x)\,dx$, $\psi\in L^2([0,1])$, is $n^{-(\beta+1/2)/(\beta+1)}$ over the H\"older ball
\[
\mathcal{C}^\beta(R)=\{g:[0,1]\rightarrow\R: |g(x)-g(y)|\le R|x-y|^\beta\,\forall
x,y\in[0,1]\}
\]
with $\beta\in(0,1]$ and radius $R>0$.  For the PPP model Rei\ss{} and Selk \cite{ReissSelk14}
build up a nonparametric maximum-likelihood approach and
construct an unbiased estimator achieving this rate.  Besides minimax
optimality, their estimator has the striking property of being UMVU (uniformly of
minimum variance among all unbiased estimators) over $\mathcal{C}^\beta(R)$.

\begin{table}
\vspace{2mm}
\begin{center}
\begin{tabular}{|l|l|l|} \hline
Rate & PPP& GWN\\\hline
 estimate $g(x)$ &  $n^{-\beta/(\beta+1)}$ & $n^{-\beta/(2\beta+1)}$\\\hline
 estimate $\langle g,\psi\rangle$ & $n^{-(\beta+1/2)/(\beta+1)}$ & $n^{-1/2}$\\\hline
estimate $\norm{g}_p^p$ & $n^{-(\beta+1/2)/(\beta+1)}$ & $p=2$: $n^{-4\beta/(4\beta+1)}\vee n^{-1/2}$\\\hline
estimate $\norm{g}_p$ & $n^{-(\beta+1/(2p))/(\beta+1)}$ & $p$ even: $n^{-\beta/(2\beta+1-1/p)}$\\\hline
testing & $n^{-(\beta+1/(2p))/(\beta+1)}$ & $n^{-\beta/(2\beta+1/2+(1/2-1/p)_+)}$\\
\hline
\end{tabular}
\end{center}
\vspace{2mm}
\caption{Minimax rates for regularity $\beta$ in the Poisson point process
(PPP) and Gaussian white noise (GWN) model.}\label{tab1}
\vspace{-5mm}
\end{table}

Here, we consider the problem of estimating and testing  non-linear functionals of the form
\begin{equation}\label{EqFunctional}
\Funct(g)=\int_0^1\Phi(g(x))\, dx,
\end{equation}
where $\Phi:\R\to\R$ is a known weakly differentiable function with derivative $\Phi'\in L_{loc}^1(\R)$ (i.e. $\Phi(u)=\Phi(0)+\int_0^u\Phi'(v)\,dv$, $u\in\R$, holds). An important class of functionals of the form \eqref{EqFunctional} is given by $p$-th powers $\norm{g}_p^p$ of $L^p$-norms using $\Phi(u)=|u|^p$, $p\ge 1$.

We show that it is still possible to construct an unbiased estimator of $\Funct(g)$ which is UMVU over $\mathcal{C}^\beta(R)$. Moreover, under weak assumptions on $\Phi'$, we compute the minimax risk of estimation over small neighbourhoods of $g$ and show that the estimator achieves the local minimax rate of convergence $\|\Phi'\circ g\|_2n^{-(\beta+1/2)/(\beta+1)}$. For the special case of estimating $\|g\|_p^p$ and the $L^p$-norm $\|g\|_p$, we prove that the minimax rates of convergence over $\mathcal{C}^\beta(R)$ are $n^{-(\beta+1/2)/(\beta+1)}$ and $n^{-(\beta+1/(2p))/(\beta+1)}$, respectively.

Based on these results we consider the testing problem $H_0:g=g_0$ versus $H_1:g\in \{g_0+h\in \mathcal{C}^\beta(R):\|h\|_p\ge r_n\}$, where the nonparametric alternative is separated by a ball of radius $r_n>0$ in $L^p$-norm. We show that the minimax separation rate is $n^{-(\beta+1/(2p))/(\beta+1)}$ and that this rate can be achieved by a plug-in test, using a minimax optimal estimator of the $L^p$-norm of $g$.
In particular, the minimax rates of testing and estimation coincide, and they are located strictly between the parametric rate $n^{-1}$ and the rate $n^{-\beta/(\beta+1)}$, corresponding to the problem of estimating the function $g$ itself (see e.g. Jirak, Meister and Rei\ss{} \cite{Moritz14} and the references therein).

These fundamental questions have been  studied extensively in the mean regression and Gaussian white noise (GWN) model. In the latter, we observe a realisation of
\[
dX(t)=g(t)\,dt+ n^{-1/2}\,dW(t),\quad t\in[0,1],
\]
where $g$ is the unknown regression function and $(W(t):t\in[0,1])$ is a standard Brownian motion. Significant differences appear. Consider, for instance, the case $\Phi(u)=|u|^p$ with $p\in\N$. For $p$ even and $\beta$ large enough, the smooth functional \eqref{EqFunctional} can be estimated with the parametric rate of convergence $n^{-1/2}$, using the method from Ibragimov, Nemirovski and Khasminski \cite{INK86} (see Table \ref{tab1} for the case $p=2$ and the monograph by Nemirovski \cite{Nem00} for more general functionals). Estimation of the $L^p$-norm has been considered by Lepski, Nemirovski and Spokoiny \cite{LepNemSpo99}. For $p$ even, the optimal rate of convergence is $n^{-\beta/(2\beta+1-1/p)}$, while for $p$ odd, the standard nonparametric rate $n^{-\beta/(2\beta+1)}$ can only be improved by $\log n$ factors. In Table~\ref{tab1} we compare these GWN estimation rates with the PPP rates. A structural difference is that  for vanishing regularity $\beta\downarrow 0$ the GWN convergence rates become arbitrarily slow, while in the PPP case the rates always remain faster than $n^{-1/2}$ and $n^{-1/(2p)}$, respectively. This phenomenon will be further discussed at the beginning of Section \ref{SecEst}. More generally, the PPP rates  hold universally for all $1\le p<\infty$, while the GWN rates depend on $p$ in a very delicate way, showing  that $L^p$-norm estimation is to some extent a regular estimation problem in the otherwise rather irregular PPP statistical model.

\begin{figure}[t]
\includegraphics[width=0.9\textwidth,height=0.25\textheight]{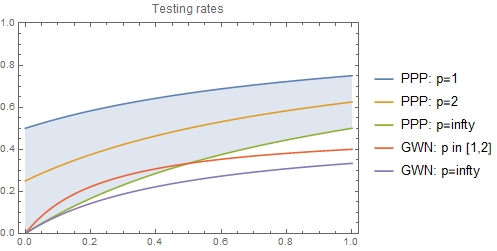}
\caption{Testing rate exponents for the Poisson point process
(PPP) and Gaussian white noise (GWN) model as a function of the regularity
$\beta$.}
\label{fig1}
\end{figure}

Further differences arise in the testing problem, which for the GWN model is the topic of the monograph by Ingster and Suslina \cite{IngsterSuslina03}.  The testing problem $H_0:g=0$ versus $H_1:g\in\{h\in L^2([0,1]):\|h\|_p\ge r_n \text{ and } \|h\|_{\beta,q}\le R\}$ is considered, where $\|\cdot\|_{\beta,q}$ is a Sobolev or Besov norm with smoothness measured in $L^q$-norm. For instance, in the case $1\le p\le 2$ and $q=\infty$, the minimax separation rate is $n^{-2\beta/(4\beta+1)}$ which coincides with the minimax rate for estimating the $L^p$-norm if $p=2$ but not if $p=1$. The general minimax GWN separation rates for the case $q\ge p$ are given in the last row of Table \ref{tab1}  (for the cases $1\le p \le 2$, $p\le q\le \infty$ and  $2<p=q<\infty$), results for the case $q<p$ can be found in Lepski and Spokoiny \cite{LepskiSpokoiny99}. Figure~\ref{fig1} visualises  the differences between the GWN and the PPP case by plotting the separation rate exponents for the range of $p\in[1,\infty)$ as a function of the regularity $\beta$.  In the GWN model the rates become arbitrarily slow when $\beta$ approaches zero and they do not change for $p\in[1,2]$ ({\it elbow effect}), which is not the case in the PPP case. The absence of an elbow effect in the PPP model may be understood by a different Hellinger geometry: the Hellinger distance is given by an $L^1$-distance between the curves, while it is based on the $L^2$-distance in the GWN model.

In the next Section \ref{SecEst} we construct the
estimator, compute its mean and variance using the underlying point process
geometry and martingale arguments, and we derive the (local) minimax rates of convergence. In Sections
\ref{SecTesting} and \ref{SecEstLp}, we focus on the special case where $\Phi(u)=|u|^p$ and apply our
results to the problem of estimating the $L^p$-norm and to the corresponding hypothesis testing problem.

\section{Estimation of non-linear functionals}\label{SecEst}

\subsection{The estimator}

Let $(X_j,Y_j)_{j\ge 1}$ be the observed support points of a Poisson point process on $[0,1]\times\R$ with intensity function given by \eqref{EqIntensity}. The support boundary curve $g$ is supposed to lie in the H\"older ball $\mathcal{C}^\beta(R)$ with $\beta\in(0,1]$. The aim is to estimate the functional in \eqref{EqFunctional}. Similarly to \cite{ReissSelk14}, our estimation method can be motivated as follows. Suppose that we know a deterministic function $\bar g\in \mathcal{C}^\beta(R)$ with $\bar
g(x)\ge g(x)$ for all $x\in[0,1]$. Then the sum
\begin{equation}\label{EqSumBar}
\frac1n\sum_{j\ge 1}\Phi'(Y_j)\mathbbm{1}\big(\bar g(X_j)\ge Y_j\big)
\end{equation}
is a.s. finite, has expectation equal to
\begin{align*}
&\frac1n\int_0^1\int_{\R}\Phi'(y)\mathbbm{1}\big(\bar g(x)\ge y\big)\lambda_g(x,y)\,dydx
=\int_0^1\big(\Phi(\bar g(x))-\Phi(g(x))\big)\,dx
\end{align*}
and variance equal to
\begin{align}
&\frac1{n^2}\int_0^1\int_{\R}\Phi'(y)^2\mathbbm{1}\big(\bar g(x)\ge y \big)\lambda_g(x,y)\,dydx\nonumber\\
&=\frac1n\int_0^1\int_{\R}\Phi'(y)^2\mathbbm{1}\big(\bar g(x)\ge y\ge g(x)\big)\,dydx\label{EqVarianceDet},
\end{align}
provided the last integral is finite (see e.g. \cite[Lemma 1.1]{Kut98} or \cite[Theorem 4.4]{LastPenrose2016}). Thus,
\[ \hat F^{pseudo}_\Phi=\int_0^1\Phi(\bar g(x))\,dx-\frac1n\sum_{j\ge 1}\Phi'(Y_j)\mathbbm{1}\big(\bar g(X_j)\ge Y_j\big)
\]
forms an unbiased pseudo-estimator (relying on the knowledge of $\bar g$) of $\Funct(g)$ whose variance is given by \eqref{EqVarianceDet}. The closer $\bar g$ is to $g$, the smaller the variance. Concerning the rate results for $p$-th powers of $L^p$-norms in Table \ref{tab1} note that already the very minor knowledge of some upper bound of $g$ suffices to construct an estimator with convergence rate $n^{-1/2}$, which explains why in the PPP case even for $\beta\downarrow 0$ estimation and testing rates remain consistent.

The main idea is now to find a data-driven upper bound of $g$ which is as small as possible. A solution to this problem is given by
\begin{equation}\label{EqMLE}
\hat g^{MLE}(x)=\min_{k\ge 1}\big(Y_k+R|x-X_k|^{\beta}\big),\qquad x\in[0,1],
\end{equation}
which is the maximum-likelihood estimator over $\mathcal{C}^\beta(R)$ \cite[Section 3]{ReissSelk14}. Indeed, $\hat g^{MLE}$ is  an upper bound for $g$ noting that $Y_k\ge g(X_k)$ and $g(X_k)+R\abs{x-X_k}^\beta\ge g(x)$ for all $k\ge 1$, where the latter follows from $g\in\mathcal{C}^\beta(R)$.

The idea is now that the sum
\[
\frac1n\sum_{j\ge 1}\Phi'(Y_j)\mathbbm{1}\big(\hat{g}^{MLE}(X_j)\ge Y_j\big)
\]
is a.s. finite and satisfies
\begin{align*}
&\E \Big[\frac1n\sum_{j\ge 1}\Phi'(Y_j)\mathbbm{1}\big(\hat{g}^{MLE}(X_j)\ge Y_j\big)\Big]\\
&=\frac1n\int_0^1\int_{\R}\Phi'(y)\E\big[\mathbbm{1}\big(\hat{g}^{MLE}(x)\ge y\big)\big]\lambda_g(x,y)\,dydx\\
&=\int_0^1\E\big[\Phi(\hat{g}^{MLE}(x))\big]\,dx-\int_0^1\Phi(g(x))\,dx,
\end{align*}
provided that the integral in the second line is well-defined. For the first equality observe that
\[
\mathbbm{1}\big(\hat{g}^{MLE}(X_j)\ge Y_j\big)=\mathbbm{1}\Big(\min_{k\ge 1:k\neq j}\big(Y_k+R|X_j-X_k|^{\beta}\big)\ge Y_j\Big)
\]
where the term $j=k$ can be dropped. This implies that the observation $(X_j,Y_j)$ can be integrated out, by following the usual arguments for computing sums with respect to a Poisson process (see e.g. \cite[Theorem 4.4]{LastPenrose2016}). To summarise, we propose the following estimator
\begin{equation}\label{EqEstimator}
 \EstGen=\int_0^1\Phi(\hat{g}^{MLE}(x))\,dx-\frac{1}{n}\sum_{j\ge 1}\Phi'(Y_j)\mathbbm{1}\big(\hat{g}^{MLE}(X_j)\ge Y_j\big),
\end{equation}
which is indeed an unbiased estimator of $\Funct(g)$ under the appropriate integrability condition.

\begin{proposition}\label{PropUnbiased} Suppose that
\begin{equation}\label{EqCondUnbiased}
\int_0^1\int_{0}^\infty|\Phi'(g(x)+u)|\PP\big(\hat{g}^{MLE}(x)-g(x)\ge u\big)\,dudx<\infty.
\end{equation}
Then $\EstGen$ from \eqref{EqEstimator} is an unbiased estimator of $\Funct(g)$.
\end{proposition}

\begin{remark}
The above argument can be worked out for more general functionals of the form $\int_0^1\dots\int_0^1\Phi(x_1,\dots,x_m,g(x_1),\dots,g(x_m))\,dx_1\dots dx_m$, but then involves complex expressions in mixed partial derivatives of $\Phi$. We therefore focus on estimation of the basic functional $\Funct$.
\end{remark}

\begin{remark}
For the one-sided regression model \eqref{EqRegression}  the discrete functional $\EstGen^{(n)}=(1/n)\sum_{i=1}^n\Phi(g(i/n))$ can be estimated analogously by
\[ \frac1n\sum_{i=1}^n\Phi(\hat g^{rMLE}(i/n))-\frac{1}{\alpha n}\sum_{i=1}^n\Phi'(Y_i)\mathbbm{1}(\hat g^{rMLE}(i/n)\ge Y_i)\]
with the regression analogue $\hat g^{rMLE}(x)=\min_i(Y_i+R\abs{x-i/n}^\beta)$ of $\hat g^{MLE}$. This estimator can be analysed with the martingale arguments of the next section, compare the results in \cite{ReissSelk14} for the linear case.
\end{remark}

\subsection{The martingale approach}
We pursue a martingale-based analysis of the estimator $\EstGen$ in \eqref{EqEstimator}. The following result extends \cite[Theorem 3.2]{ReissSelk14} to non-linear functionals.

\begin{thm}\label{ThmGenRes}  Suppose that the right-hand side in \eqref{EqPhiVar} below is finite. Then the estimator $\EstGen$ is UMVU over $g\in\mathcal{C}^\beta(R)$ with variance
\begin{equation}\label{EqPhiVar}
\operatorname{Var}(\EstGen)=\frac{1}{n}\int_0^1\int_{0}^\infty(\Phi'(g(x)+u))^2\PP\big(\hat{g}^{MLE}(x)-g(x)\ge u\big)\,dudx.
\end{equation}
\end{thm}
\begin{remark}
If the right-hand side in \eqref{EqPhiVar} is finite, then Condition \eqref{EqCondUnbiased} holds since $\PP(\hat g^{MLE}(x)-g(x)\ge u)$ is integrable in $u$, see also \eqref{EqTails} below.
\end{remark}
\begin{proof}
We first show the formula for the variance.
Let $\lambda=(\lambda_t)$ be the process defined by $\lambda_t=n\int_0^1 \mathbbm{1}\big(g(x)\le t\le \hat{g}^{MLE}(x)\big)\,dx$, $t\in\R$.
Making a linear change of variables, the right-hand side in \eqref{EqPhiVar} can be written as
\begin{equation*}
n^{-2}\E\Big[\int_{t_0}^\infty\Phi'(s)^2\lambda_s\,ds\Big],
\end{equation*}
where $t_0$ is a lower bound for $g$.
In the proof of Theorem 3.2 in \cite{ReissSelk14}, it is shown that the pure counting process $N=(N_t)$ defined by
\[
N_t=\sum_{j\ge 1}\mathbbm{1}\big(Y_j\le t\wedge  \hat{g}^{MLE}(X_j)\big),\quad t\ge t_0,
\]
has compensator $A=(A_t)$ given by $A_t=\int_{t_0}^t\lambda_s\,ds$
and that $M=N-A$ is a square-integrable martingale with respect to the filtration ${\mathcal F}_t=\sigma((X_j,Y_j)\mathbbm{1}(Y_j\le t), j\ge 1)$. Its predictable quadratic variation is
\[
\langle M \rangle_t=\int_{t_0}^t\lambda_s\,ds
\]
(see also \cite[Proposition 2.32]{Karr91}). We conclude (e.g. via \cite[Theorem 26.2]{Kallenberg02}) that
\[
(\Phi'\cdot M)_t=\int_{t_0}^t\Phi'(s)\,dM_s=\sum_{j\ge 1}\Phi'(Y_j)\mathbbm{1}\big(Y_j\le t\wedge  \hat{g}^{MLE}(X_j)\big)-\int_{t_0}^t\Phi'(s)\lambda_s\,ds
\]
is an $L^2$-bounded martingale with
\[
\langle \Phi'\cdot M \rangle_t=\int_{t_0}^t\Phi'(s)^2\lambda_s\,ds,
\]
noting that $\E[\langle \Phi'\cdot M \rangle_t]$ is bounded by the right-hand side in \eqref{EqPhiVar}, which is finite by assumption.
For $t\to\infty$ the process $((\Phi'\cdot M)_t)$ converges almost surely to
\begin{align*}
&(\Phi'\cdot M)_\infty\\
&=\sum_{j\ge 1}\Phi'(Y_j)\mathbbm{1}\big(Y_j\le  \hat{g}^{MLE}(X_j)\big)-\int_{t_0}^\infty\Phi'(s)\lambda_s\,ds\\
&=\sum_{j\ge 1}\Phi'(Y_j)\mathbbm{1}\big(Y_j\le  \hat{g}^{MLE}(X_j)\big)-n\int_0^1\Phi(\hat{g}^{MLE}(x))\,dx+n\int_0^1\Phi(g(x))\,dx.
\end{align*}
Moreover, the process $ (\langle \Phi'\cdot M \rangle_t)$ converges almost surely and in $L^1$ to
\[
\langle \Phi'\cdot M \rangle_\infty=\int_{t_0}^\infty\Phi'(s)^2\lambda_s\,ds.
\]
Hence, unbiasedness and \eqref{EqPhiVar} follow from
\begin{equation}\label{EqMartingaleForm}
\E[(\Phi'\cdot M)_\infty]=0\ \ \text{ and }\ \ \E[(\Phi'\cdot M)_\infty^2-\langle \Phi'\cdot M \rangle_\infty]=0,
\end{equation}
which holds due to the $L^2$-convergence of $\Phi'\cdot M$ \cite[Corollary 6.22]{Kallenberg02}.

Finally, the fact that $\EstGen$ is UMVU follows from the Lehmann-Scheffé theorem and \cite[Proposition 3.1]{ReissSelk14} which says that $(\hat{g}^{MLE}(x):x\in[0,1])$ is a sufficient and complete statistic for $\mathcal{C}^\beta(R)$.
\end{proof}

\subsection{Rates of convergence}

In this section, we derive convergence rates for the estimator $\EstGen$. Using the argument leading to \cite[Equation (3.3)]{ReissSelk14}, we have the following deviation inequality for $x\in[0,1]$:
\begin{equation}\label{EqTails}
\mathbb{P}\big(\hat{g}^{MLE}(x)-g(x)\ge u\big)\le
\begin{cases}
\exp\big(-\frac{n\beta(2R)^{-\frac{1}{\beta}}u^{\frac{\beta+1}{\beta}}}{\beta+1}\big),& \text{if }u\in[0,2R],\\
\exp\big(-n\big(u-\frac{2R}{\beta+1}\big)\big),&\text{if }u>2R.
\end{cases}
\end{equation}
Thus, the right-hand side in \eqref{EqPhiVar} is finite if $(\Phi')^2$ has at most exponential growth with parameter strictly smaller than $n$. In particular, this holds for $\Phi(u)=|u|^p$, $p\ge 1$, in which case we have $\Phi'(u)=p|u|^{p-1}\operatorname{sgn}(u)$. A more detailed analysis gives:

\begin{corollary}\label{CorPowInt} Let $p\ge 1$ be a real number and consider $\Phi(u)=|u|^p$, $g\in\mathcal{C}^\beta(R)$. Then
\begin{equation}\label{EqFhatLp} \EstNorm=\int_0^1|\hat{g}^{MLE}(x)|^p\,dx-\frac{1}{n}\sum_{j\ge 1}p|Y_j|^{p-1}\sgn(Y_j)\mathbbm{1}\big(\hat{g}^{MLE}(X_j)\ge Y_j\big)
\end{equation}
is an unbiased estimator of $\| g\|_p^p$ with
\begin{equation}\label{EqPowInt}
\mathbb{E}\big[(\EstNorm-\| g\|_p^p)^2\big]\le C\max\Big(\|g\|_{2p-2}^{2p-2}n^{-\frac{2\beta+1}{\beta+1}},n^{-\frac{2\beta p+1}{\beta+1}}\Big),
\end{equation}
where $C$ is a constant depending only on $R$, $\beta$ and $p$. Here, we use the notation $\|\cdot\|_q$ also for $q<1$ with $\|g\|_{0}^{0}:=1$.
\end{corollary}

\begin{remark}
In the proof, a more precise upper bound is derived in which the dependence on the constant $R$ is explicit, see \eqref{EqBinomFormula}. For an asymptotically more precise result see Corollary \ref{CorFrechet} below.
\end{remark}

\begin{remark}\label{RemPosPart}
Since $\Phi(u)=|u|^p$ is non-negative, the positive part $(\EstNorm)_+$ of $\EstNorm$ always improves the estimator. This means that $\EstNorm$ is not an admissible estimator in the decision-theoretic sense, while $(\EstNorm)_+$ on the other hand is no longer unbiased.
\end{remark}

\begin{proof}
Throughout the proof $C>0$ denotes a constant depending only on $\beta$ and $p$ that may change from line to line.
By Theorem \ref{ThmGenRes} and the discussion above, we have
\[
\mathbb{E}\big[(\EstNorm-\| g\|_p^p)^2\big]=\frac{1}{n}\int_0^1\int_0^\infty p^2|u+g(x)|^{2p-2}\mathbb{P}\big(\hat{g}^{MLE}(x)-g(x)\ge u\big)\,dudx.
\]
Applying \eqref{EqTails} and the inequality $|u+g(x)|^{2p-2}\le 2^{2p-2}(u^{2p-2}+|g(x)|^{2p-2})$, the last term is bounded by
\begin{align*}
&\frac{p^22^{2p-2}}{n}\int_0^{2R}(\|g\|_{2p-2}^{2p-2}+u^{2p-2})\exp\Big(-\frac{n\beta(2R)^{-\frac{1}{\beta}}u^{\frac{\beta+1}{\beta}}}{\beta+1}\Big)\,du\\
&+\frac{p^22^{2p-2}}{n}\int_{2R}^\infty (\|g\|_{2p-2}^{2p-2}+u^{2p-2})\exp\Big(-n\Big(u-\frac{2R}{\beta+1}\Big)\Big)\,du=:(I)+(II).
\end{align*}
By a linear substitution, we have for $q\ge 0$
\begin{align}
&\int_0^{2R}u^q\exp\Big(-\frac{n\beta(2R)^{-\frac{1}{\beta}}u^{\frac{\beta+1}{\beta}}}{\beta+1}\Big)\,du\nonumber\\
&\le\Big(\frac{\beta+1}{\beta}\Big)^{\frac{\beta(q+1)}{\beta+1}}(2R)^{\frac{q+1}{\beta+1}}n^{-\frac{\beta(q+1)}{\beta+1}}\int_0^{\infty}v^{q}\exp\big(-v^{\frac{\beta+1}{\beta}}\big)\,dv\nonumber\\
&= \Big(\frac{\beta+1}{\beta}\Big)^{\frac{\beta q-1}{\beta+1}}(2R)^{\frac{q+1}{\beta+1}}\Gamma\Big(\frac{\beta(q+1)}{\beta+1}\Big)n^{-\frac{\beta(q+1)}{\beta+1}}\label{EqPowIntAsympLimit}
\end{align}
with the Gamma function $\Gamma$. Consequently,
\[
(I)\le  CR^{\frac{1}{\beta+1}}\|g\|_{2p-2}^{2p-2}n^{-\frac{2\beta+1}{\beta+1}}+CR^{\frac{2p-1}{\beta+1}}n^{-\frac{2\beta p+1}{\beta+1}}.
\]
Next, consider the remainder term $(II)$. We have
\[
\int_{2R}^\infty \exp\Big(-n\Big(u-\frac{2R}{\beta+1}\Big)\Big)\,du= n^{-1}e^{-\frac{2\beta Rn}{\beta+1}}
\]
and
\begin{align*}
&\int_{2R}^\infty u^{2p-2}\exp\Big(-n\Big(u-\frac{2R}{\beta+1}\Big)\Big)\,du\le \int_{2R}^\infty u^{2p-2}\exp\Big(-\frac{n\beta u}{\beta+1}\Big)\,du\\
&\le  \Big(\frac{\beta+1}{n\beta}\Big)^{2p-1}\int_{2\beta Rn/(\beta+1)}^\infty v^{2p-2}\exp(-v)\,dv\le  C n^{-2p+1}e^{-\frac{\beta Rn}{\beta+1}}.
\end{align*}
Note that the last integral can be computed using partial integration. Thus
\[
(II)\le  C\|g\|_{2p-2}^{2p-2}n^{-2}e^{-\frac{2\beta Rn}{\beta+1}}+Cn^{-2p}e^{-\frac{\beta Rn}{\beta+1}}.
\]
Summarising, we have
\begin{align}
\mathbb{E}\big[(\EstNorm-\| g\|_p^p)^2\big]&\le  CR^{\frac{1}{\beta+1}}\|g\|_{2p-2}^{2p-2}n^{-\frac{2\beta+1}{\beta+1}}+CR^{\frac{2p-1}{\beta+1}}n^{-\frac{2\beta p+1}{\beta+1}}\nonumber\\
&+C\|g\|_{2p-2}^{2p-2}n^{-2}e^{-\frac{2\beta Rn}{\beta+1}}+Cn^{-2p}e^{-\frac{\beta Rn}{\beta+1}},\label{EqBinomFormula}
\end{align}
and the claim follows.
\end{proof}

One might wonder whether $\EstNorm$ achieves the rate $n^{-(\beta+1/2)/(\beta+1)}$ uniformly over $g\in{\mathcal C}^\beta(R)\cap B_p(R)$ with the $L^p$-ball $B_p(R)=\{g\in L^p([0,1]):\norm{g}_p\le R\}$. For $1\le p\le 2$ this follows from the inclusion $B_p(R)\subseteq B_{2p-2}(R)$. For $p>2$ this holds as well and is a consequence of
the following useful Lemma (with $q=2p-2$) providing a simple interpolation result. Results of this type are well known (cf. Bergh and L\"ofstr\"om \cite{BL76}), but since only H\"older semi-norms appear, we provide a self-contained proof in the appendix.

\begin{lemma}\label{LemNormComp} Let $1\le p\le q\le \infty$ and $f\in \mathcal{C}^\beta(R)$. Then we have
\begin{equation*}
\|f\|_q\le C\|f\|_p\max(1,R/\|f\|_{p})^{\frac{1/p-1/q}{\beta+1/p}},
\end{equation*}
where $C>0$ is a constant depending only on $\beta$, $p$ and $q$ and the right-hand side is understood to be zero for $f=0$.
\end{lemma}

Let us come to another corollary of Theorem \ref{ThmGenRes} which provides a local asymptotic upper bound for the minimax risk under weak assumptions on the functional:

\begin{corollary}\label{CorFrechet} Suppose that there is a constant $C>0$ such that $|\Phi'(u)|\le C\exp(C|u|)$ for all $u\in\R$. Let $f\in \mathcal{C}^\beta(R)$. Suppose that $\norm{\Phi'\circ f}_2\neq 0$ and that the map $F':\mathcal{C}^\beta(R)\subset L^2([0,1])\rightarrow L^2([0,1])$, $F'(g)= \Phi'\circ g$ is continuous at $g=f$ with respect to the $L^2$-norms. Then the estimator $ \EstAsymp=\EstGen$ satisfies the local asymptotic upper bound
\begin{equation*}
\lim_{\delta\rightarrow 0}\limsup_{n\rightarrow\infty}\sup_{\substack{g\in\mathcal{C}^\beta(R):\\ \|f-g\|_2\le \delta}}n^{\frac{2\beta+1}{\beta+1}}\mathbb{E}_g\big[(\EstAsymp-\Funct(g))^2\big]\le \Gamma\big(\tfrac{\beta}{\beta+1}\big)\big(\tfrac{2R\beta}{\beta+1}\big)^{\frac{1}{\beta+1}}\|\Phi'\circ f\|_2^2
\end{equation*}
with the Gamma function $\Gamma$.
\end{corollary}

\begin{proof}
By Theorem \ref{ThmGenRes} and Equation \eqref{EqTails}, we have
\begin{align*}
\mathbb{E}_g\big[(\EstGen-\Funct(g))^2\big]
&\le \frac{1}{n}\int_0^{2R}\|\Phi'\circ(u+g)\|_2^2\exp\Big(-\frac{n\beta(2R)^{-\frac{1}{\beta}}u^{\frac{\beta+1}{\beta}}}{\beta+1}\Big)\,du\\
&+\frac{1}{n}\int_{2R}^\infty\|\Phi'\circ(u+g)\|_2^2\exp\Big(-\frac{n\beta u}{\beta+1}\Big)\,du.
\end{align*}
By Lemma \ref{LemNormComp}, applied to $f-g$ and with $p=2$, $q=\infty$, we infer from $g\in\mathcal{C}^\beta(R)$ with $\|f-g\|_2\le \delta$ that
\begin{equation}\label{EqSupBound}
\|f-g\|_\infty\le C'R^{1/(2\beta+1)}\delta^{2\beta/(2\beta+1)}
\end{equation}
holds with some constant $C'$, provided that $\delta\le R$.
Using that $\Phi'$ has at most exponential growth, we get that $\|\Phi'\circ(u+g)\|_2^2\le C\exp(C|u|)$ uniformly over all $g\in\mathcal{C}^\beta(R)$ with $\|f-g\|_2\le \delta$ (adjusting $C$ appropriately). This shows that the second term is of smaller order than $n^{-2}$ and thus asymptotically negligible for our result. Similarly, for every fixed $\delta'>0$ the first integral from $\delta'$ to $2R$ becomes exponentially small in $n$.
Thus, for any $\delta'>0$ the left-hand side in Corollary \ref{CorFrechet} is bounded by
\begin{equation}\label{EqCorFrechetPr}
\lim_{\delta\to 0}\limsup_{n\rightarrow\infty}\sup_{\substack{g\in\mathcal{C}^\beta(R):\\ \|f-g\|_2\le \delta}}n^{\frac{\beta}{\beta+1}} \int_0^{\delta'}\|\Phi'\circ(u+g)\|_2^2\exp\Big(-\frac{n\beta(2R)^{-\frac{1}{\beta}}
u^{\frac{\beta+1}{\beta}}}{\beta+1}\Big)\,du.
\end{equation}
By the continuity of $\Funct'$ at $f$ and the fact that $\norm{\Phi'\circ f}_2\neq 0$, for every $\eps>0$ there exist $\delta,\delta'>0$ such that $\|\Phi'\circ(u+g)\|_2\le (1+\eps)\|\Phi'\circ f\|_2$ for all $|u|\le \delta'$ and $g\in\mathcal{C}^\beta(R)$ with $\|f-g\|_2\le \delta$. We conclude that \eqref{EqCorFrechetPr} is bounded by (using the computation in \eqref{EqPowIntAsympLimit} for $q=0$)
\[
\Gamma\Big(\frac{\beta}{\beta+1}\Big)\Big(\frac{2R\beta}{\beta+1}\Big)^{\frac{1}{\beta+1}}\|\Phi'\circ f\|_2^2,
\]
and the claim follows.
\end{proof}

\begin{remark}
By Lemma \ref{LemNormComp} continuity of $\Funct'$ on $\mathcal{C}^\beta(R)$ with respect to $L^2$-norm implies continuity with respect to supremum norm. Under the assumptions of Corollary \ref{CorFrechet}, one can indeed show that the functional $\Funct$ is Fr\'{e}chet-differentiable in $f$ along $\mathcal{C}^\beta(R)$ with derivative $\Funct'(f)=\Phi'\circ f$.
\end{remark}

\begin{remark}
Local asymptotic minimax results for estimating smooth functionals in the GWN model can be found in Nemirovski \cite[Chapter 7]{Nem00}. The rate is different (see the discussion in the introduction), but the term $\|\Funct'(g)\|_2^2$ appears as well. The latter fact can be explained by linearising $\Funct$ at~$g$.
\end{remark}

\begin{remark}
The estimators are non-adaptive in the sense that they rely on the knowledge of the regularity parameters $\beta$ and $R$. In \cite{ReissSelk14} the Lepski method has been employed to construct adaptive estimators in the linear case, based on a blockwise estimator. We conjecture that this approach would also give an adaptive rate-optimal estimator here. Note also the restriction $\beta\le 1$ on the regularity parameter. The  reason is that the MLE for $g\in C^\beta(R)$ with $\beta>1$ does not necessarily provide a pointwise upper bound for $g$ such that the present approach may fail. 
\end{remark}

\subsection{Lower bounds}

In this section we establish lower bounds corresponding to Corollaries \ref{CorPowInt} and \ref{CorFrechet}. We will apply the method of two fuzzy hypotheses (see \cite[Chapter 2.7.4]{Tsyb09}) with a prior corresponding to independent non-identical Bernoulli random variables. Our main result states a local asymptotic lower bound in the case that $\Phi$ is continuously differentiable. Possible extensions are discussed afterwards.

\begin{thm}\label{PropLowBoundDiff}
Let $\Phi$ be continuously differentiable and $f\in \mathcal{C}^\beta(R)$ with $\norm{\Phi'\circ f}_2\neq 0$. Then there is a  constant $c_1>0$, depending only on $\beta$, such that
\begin{equation*}
\lim_{\delta\rightarrow 0}\liminf_{n\rightarrow\infty}\inf_{\tilde{F}_n}\sup_{\substack{g\in\mathcal{C}^\beta(R):\\ \|f-g\|_2\le \delta}}n^{\frac{2\beta+1}{\beta+1}}\mathbb{E}_g\big[(\tilde{F}_n -\Funct(g))^2\big]> c_1R^{\frac{1}{\beta+1}}\|\Phi'\circ f\|_2^2.
\end{equation*}
The infimum is taken over all estimators in the PPP model with intensity \eqref{EqIntensity}.
\end{thm}

\begin{proof}
We want to apply the $\chi^2$-version of the method of two fuzzy hypotheses as described in \cite[Theorem 2.15]{Tsyb09}. Consider the functions
\[
g_\theta=\sum_{k=1}^{m}\theta_kg_k\ \ \text{ with } \ \ \theta_k\in\{0,1\}
\]
and \[
g_k(x)=cRh^{\beta}K\left(\frac{x-(k-1)h}{h}\right)=cRh^{\beta+1}K_h(x-(k-1)h)
\]
with $h=1/m$, triangular kernel $K(u)=4(u\wedge (1-u))\mathbbm{1}_{[0,1]}(u)$, $K_h(\cdot)=K(\cdot/h)/h$ and $c>0$ sufficiently small such that $g_\theta\in\mathcal{C}^\beta(R)$ for all $m$ and $\theta$. Let $\pi_n$ be the probability measure on $\{0,1\}^m$ obtained when $\theta_1,\dots,\theta_m$ are independent (non-identical) Bernoulli random variables with success probabilities $p_1,\dots,p_m$. Let $P_g$ denote the law of the observations in the PPP model with intensity function \eqref{EqIntensity}. We set $\mathbf{P}_{0,n}=P_f$ and
\[
\mathbf{P}_{1,n}(\cdot)=\int P_{f+g_\theta}(\cdot)\pi_n(d\theta).
\]
In order to obtain the result, it suffices to find $m\ge 1$ and probabilities $p_1,\dots,p_m$ (both depending on $n$) as well as a constant $c_1>0$, only depending on $\beta$, and an absolute constant $c_2<\infty$, such that
\begin{enumerate}
\item[(i)] For each fixed $\delta>0$ the inequality $\|g_\theta\|_2\le \delta$ holds for all $n$ sufficiently large and for $n\to\infty$ the prior satisfies 
\begin{equation*}
\pi_n\Big( \Funct(f+g_\theta)\ge \Funct(f)+2c_1\|\Phi'\circ f\|_2R^{\frac{1/2}{\beta+1}}n^{-\frac{\beta+1/2}{\beta+1}}\Big)\rightarrow 1;
\end{equation*}

\item[(ii)] $\limsup_{n\to\infty} \chi^2(\mathbf{P}_{1,n},\mathbf{P}_{0,n})\le c_2$.
\end{enumerate}
We start with the following lemma on the $\chi^2$-distance.
\begin{lemma}\label{LemChiSquare} Suppose that the success probabilities satisfy $\sum_{k=1}^m p_k^2=1$. Then
\begin{equation*}
\chi^2(\mathbf{P}_{1,n},\mathbf{P}_{0,n})= \int\left(\frac{d\mathbf{P}_{1,n}}{d\mathbf{P}_{0,n}}\right)^2\,d\mathbf{P}_{0,n}-1\le
\exp\bigg(\exp\Big( n\int_{I_1}g_1(x)\,dx\Big)-1\bigg)-1
\end{equation*}
holds, where $I_1=[0,h)$.
\end{lemma}

\begin{proof}[Proof of Lemma \ref{LemChiSquare}]
We abbreviate $\int g_k=\int_{I_k}g_k(x)\,dx$, where $I_k=[(k-1)h,kh)$ for $k<m$ and $I_m=[1-h,1]$. Let us first see that
\begin{equation}\label{EqLikelihoodFuzzyH}
\frac{d\mathbf{P}_{1,n}}{d\mathbf{P}_{0,n}}=\prod_{k=1}^m\Big(1-p_k+p_ke^{n\int
g_k} \mathbbm{1}\big(\forall X_j\in I_k:Y_j\ge f(X_j)+g_k(X_j)\big)\Big).
\end{equation}
Indeed, by definition the left hand side is equal to
\begin{align*}
&\sum_{\theta\in\{0,1\}^m}\bigg(\prod_{k:\theta_k=0}(1-p_k)\prod_{k:\theta_k=1}p_k\bigg)\frac{dP_{f+g_\theta}}{dP_{f}}\\
&=\sum_{\theta\in\{0,1\}^m}\bigg(\prod_{k:\theta_k=0}(1-p_k)\prod_{k:\theta_k=1}p_ke^{n\int
g_k} \mathbbm{1}\big(\forall X_j\in I_k:Y_j\ge f(X_j)+g_k(X_j)\big)\bigg)\\
&=\prod_{k=1}^m\Big(1-p_k+p_ke^{n\int
g_k} \mathbbm{1}\big(\forall X_j\in I_k:Y_j\ge f(X_j)+g_k(X_j)\big)\Big),
\end{align*}
where we used the formula (see \cite[Theorem 1.3]{Kut98} or \cite[Section 3]{ReissSelk14})
\begin{align*}
\frac{dP_{f+g_\theta}}{dP_f}&=e^{n\int g_\theta}\mathbbm{1}\big(\forall
j:Y_j\ge f(X_j)+g_\theta(X_j)\big)\\
&=\prod_{k:\theta_k=1}e^{n\int g_k}\mathbbm{1}\big(\forall X_j\in I_k:Y_j\ge f(X_j)+g_k(X_j)\big)
\end{align*}
in the first equality. By the defining properties of the PPP, under $\mathbf{P}_{0,n}$, the right-hand side in \eqref{EqLikelihoodFuzzyH} is a product of independent random variables and the corresponding indicators have success probabilities $e^{-n\int g_k}$. Thus we obtain
\begin{align*}
\int\left(\frac{d\mathbf{P}_{1,n}}{d\mathbf{P}_{0,n}}\right)^2\,d\mathbf{P}_{0,n}&=\prod_{k=1}^m((1-p_k)^2+2p_k(1-p_k)+p_k^2e^{n\int
g_k})\\
&=\prod_{k=1}^m(1+p_k^2(e^{n\int g_k}-1))\\
&\le \prod_{k=1}^me^{p_k^2(e^{n\int g_1}-1)} = e^{e^{n\int g_1}-1},
\end{align*}
where we used the bound $1+x\le e^x$ and the assumption $\sum_{k=1}^m p_k^2=1$.
\end{proof}

Using Lemma \ref{LemChiSquare} and the identity
\[
n\int_{I_1}g_1(x)\,dx=cRnh^{\beta+1},
\]
we get (ii) provided that we choose $m=1/h$ of size $(Rn)^{1/(\beta+1)}$ and $p_1,\dots,p_m$ such that $\sum_{k=1}^mp_k^2=1$. Thus it remains to choose the $p_k$ such that the second convergence in (i) is satisfied.

We first consider the  case that $\Phi'\circ f\ge 0$.  Let $\eps>0$ be a small constant to be chosen later. Since $\Phi'$ is uniformly continuous on compact intervals, there is a $\delta'>0$ such that
\[
\int_0^1\Phi(f(x)+g(x))\,dx- \int_0^1\Phi(f(x))\,dx\ge \int_0^1\Phi'(f(x))g(x)\,dx-\eps\int_0^1|g(x)|\,dx
\]
for all $g\in\mathcal{C}^\beta(R)$ with $\|f-g\|_2\le \delta'$ (using \eqref{EqSupBound} above). Thus, for $n$ sufficiently large, we get
\begin{align*}
\Funct(f+g_\theta)-\Funct(f)&\ge \langle \Phi'\circ f,g_\theta\rangle-\eps\langle 1,g_\theta\rangle\\
& =\sum_{k=1}^m \theta_k\langle \Phi'\circ f,g_k\rangle-\eps \sum_{k=1}^m \theta_k\langle 1,g_k\rangle\\
&=cRh^{\beta+1}\bigg(\sum_{k=1}^m \theta_k\langle \Phi'\circ f,K_h(\cdot-(k-1)h)\rangle-\eps \sum_{k=1}^m \theta_k\bigg).
\end{align*}
Setting $a_k=\langle \Phi'\circ f,K_h(\cdot-(k-1)h)\rangle$, this can be written as
\begin{equation}\label{EqSepDiff}
\Funct(f+g_\theta)-\Funct(f)\ge cRh^{\beta+1}\bigg(\sum_{k=1}^m a_k\theta_k-\eps \sum_{k=1}^m \theta_k\bigg),
\end{equation}
The first sum is a weighted sum of independent non-identical Bernoulli random variables and the maximising choice for the success probabilities is
\begin{equation}\label{EqChoicePrior}
p_k=\frac{a_k}{\|a\|_2}
\end{equation}
(the $a_k$ satisfy $a_k\ge 0$  since we assumed $\Phi'\circ f\ge 0$).
By the mean value theorem and the fact that $\Phi'\circ f$ is continuous, we get $a_k=\Phi'(f(x_k))$ with $x_k\in [(k-1)h,kh]$ and also
\begin{equation}\label{EqIntConv}
\frac{1}{m}\|a\|_q^q=\frac{1}{m}\sum_{k=1}^m a_k^q\rightarrow \int_0^1 (\Phi'(f(x))^q\,dx=\|\Phi'\circ f\|_q^q\qquad\text{as } n\rightarrow \infty
\end{equation}
for each $q\ge 1$. Using the Chebyshev inequality we get
\begin{align*}
\pi_n\left(\sum_{k=1}^m a_k\theta_k< \|a\|_2/2\right)&= \pi_n\left(\sum_{k=1}^m a_k(\theta_k-p_k)< -\|a\|_2/2\right)\\
&\le \frac{4\sum_{k=1}^ma_k^2p_k(1-p_k)}{\|a\|_2^2}\le  4\left(\frac{\|a\|_3}{\|a\|_2}\right)^3
\end{align*}
and the latter converges to $0$ as $n\rightarrow\infty$ by \eqref{EqIntConv}. Similarly,
\begin{align*}
\pi_n\left(\sum_{k=1}^m\theta_k> 2\|a\|_1/\|a\|_2\right)
&=\pi_n\left(\sum_{k=1}^m(\theta_k-p_k)> \|a\|_1/\|a\|_2\right)\\
&\le \frac{\|a\|_2^2\sum_{k=1}^mp_k(1-p_k)}{\|a\|_1^2}\le \frac{\|a\|_2}{\|a\|_1}
\end{align*}
and the latter converges to $0$ as $n\rightarrow\infty$ by \eqref{EqIntConv}. Combining these two bounds with \eqref{EqSepDiff} we get
\begin{equation}\label{EqAsympSep}
\pi_n\left( \Funct(f+g_\theta)-\Funct(f)\ge cRh^{\beta+1/2}\left(\frac{1}{2\sqrt{m}}\|a\|_2-\eps\frac{2}{\sqrt{m}}\frac{\|a\|_1}{\|a\|_2}\right)\right)\rightarrow 1
\end{equation}
as $n\rightarrow\infty$. This implies (i) if $\eps$ is chosen small enough since $\|a\|_2/\sqrt{m}$ and $\|a\|_1/(\sqrt{m}\|a\|_2)$ have non-zero limits by \eqref{EqIntConv} and the assumption $\|\Phi'\circ f\|_2\neq 0$. This completes the proof in the case  $\Phi'\circ f\ge 0$.

 If $\Phi'\circ f\le 0$, then we may follow the same line of arguments where (ii) is replaced with a left-deviation inequality (which corresponds to apply the above arguments to the functional $F_{-\Phi}$). Next, if $\Phi'\circ f$ takes both, positive and negative values, then we may choose $p_k=a_{k+}/\|a_+\|_2$ (resp. $p_k=a_{k-}/\|a_-\|_2$) leading to a lower bound with $\|\Phi'\circ f\|_2^2$ replaced by $\|(\Phi'\circ f)_+\|_2^2$ (resp. $\|(\Phi'\circ f)_-\|_2^2$). Summing up both lower bounds gives the claim in the general case.
\end{proof}

\begin{remark}\label{RemConvexity} If $\Phi$ is convex, then we can replace \eqref{EqSepDiff} by
\[
\Funct(f+g_\theta)-\Funct(f)\ge \langle \Phi'\circ f,g_\theta\rangle=cRh^{\beta+1}\sum_{k=1}^m a_k\theta_k,
\]
leading to a shortening of the above proof.
In this case the lower bound also holds without continuity of $\Phi'$. The arguments, however, must be adapted slightly  since the convergence in \eqref{EqIntConv} may not hold in this case.
\end{remark}

\begin{remark}\label{RemLowBoundNonasymp}
By making the constants in the proof of Theorem \ref{PropLowBoundDiff} explicit, one can also establish non-asymptotic lower bounds which include lower-order terms. Consider for instance $\Phi(u)=|u|^p$, $p\in\N$ and $f\equiv a>0$. Then we have
\begin{align}
\Funct(a+g_\theta)-a^p&
=\bigg(\sum_{k=1}^{m}\theta_k\bigg)\sum_{j=1}^p\binom{p}{j}a^{p-j}c^jR^jh^{\beta j +1}\|K\|_j^j\nonumber\\
&\ge\bigg(\sum_{k=1}^{m}\theta_k\bigg) \max(pa^{p-1}cR h^{\beta+1},c^pR^p\|K\|_p^ph^{\beta p+1})\label{EqSepPol}.
\end{align}
We choose
\[
p_1=\dots=p_m=1/\sqrt{m}\ \ \text{ and } \ \ m=\lfloor 2(cRn)^{1/(\beta+1)}\rfloor.
\]
 In order to ensure $\|g_\theta\|_2 \le \delta$, it suffices that $m\ge 1$ and $2cRh^{\beta}\le\delta$ hold, which is satisfied if $n\ge c_1$ with $c_1$ depending only on $c$, $R$ and $\delta$. Now, by Lemma \ref{LemChiSquare} and the choice of $m$ we have $\chi^2(\mathbf{P}_{0,n},\mathbf{P}_{1,n})\le e^{e-1}-1$. Moreover, using the simplification of Remark \ref{RemConvexity}, \eqref{EqAsympSep} becomes
\begin{equation*}
\pi_n\Big( \Funct(a+g_\theta)\ge a^p+\frac{1}{2}\max(pa^{p-1}cR h^{\beta+1/2},c^pR^p\|K\|_p^ph^{\beta p+1/2})\Big)\ge 1-4/\sqrt{m}.
\end{equation*}
Inserting the value of $h$ and applying \cite[Theorem 2.15 (iii)]{Tsyb09}, we get
\begin{align*}
&\inf_{\tilde{F}_n}\sup_{\substack{g\in\mathcal{C}^\beta(R):\\ \|a-g\|_2\le \delta}}\mathbb{P}_g\Big(|\tilde{F}_n-\Funct(g)|\ge \max(c_2pa^{p-1}R^{\frac{1/2}{\beta+1}}n^{-\frac{\beta+1/2}{\beta+1}},c_3R^{\frac{p-1/2}{\beta+1}}n^{-\frac{\beta p+1/2}{\beta+1}})\Big)\\
& \ge \frac{1}{4}\exp(-(e^{e-1}-1))-2/\sqrt{m},
\end{align*}
provided that $n\ge c_1$, where $c_2$ is a constant depending only on $\beta$ and  $c_3$ is a constant depending only on $\beta$ and $p$. Thus we obtain a lower bound which has the form of the upper bound in Corollary \ref{CorPowInt} (resp. \eqref{EqBinomFormula}).
\end{remark}

\begin{remark}
In the case of linear functionals the above proof can be used to obtain the lower bound in \cite[Theorem 2.6]{ReissSelk14}. Instead of using the method of fuzzy hypothesis, one can also try to apply the method used in Rei\ss{} and Selk \cite{ReissSelk14} and Korostelev and Tsybakov \cite{KorTsy93} which is based on a comparison of the minimax risk with a Bayesian risk. This works for instance for the special case $\Phi(u)=|u|^p$, $p\in\N$, and $f\equiv a>0$, but it is not clear whether this structurally different prior can produce the correct lower bounds more generally.
\end{remark}

\section{Hypothesis testing}
\label{SecTesting}

\subsection{Main result}
In this section we use the previous results to address the hypothesis testing problem
\begin{equation*}
H_0:g=g_0\qquad\text{vs.}\qquad H_1:g\in g_0+\mathcal{G}_n,
\end{equation*}
where $g_0$ is a known function and
\[
\mathcal{G}_n=\mathcal{G}_n(\beta,R,p,r_n)=\{g\in \mathcal{C}^\beta(R):\|g\|_p\ge r_n\}.
\]
In the sequel, we restrict to the case $g_0=0$, since the general case can be reduced to this one by a simple shift of the observations.
We  propose the following plug-in test
\begin{equation}\label{EqPlugInTest}
\psi_{n,p} = \mathbbm{1}\big( \EstNorm \ge r^p_n/2\big),
\end{equation}
with the estimator $\EstNorm$ from \eqref{EqFhatLp}.
We follow a minimax approach to hypothesis testing, see e.g. \cite[Chapter 2.4]{IngsterSuslina03}.
Our main result of this section states that $\psi_{n,p}$ achieves the minimax separation rates:

\begin{thm}\label{TheoremTesting}
Let $p\ge 1$ be a real number and
\[
r_n^*=n^{-\frac{\beta+1/(2p)}{\beta+1}}.
\]
Then, the following holds as $n\rightarrow\infty$:
\begin{itemize}
\item[(a)] If $r_n/r^*_n\rightarrow \infty$, then the tests
$\psi_{n,p}$ from \eqref{EqPlugInTest} satisfy
 \[
\mathbb{E}_0[\psi_{n,p}] + \sup_{g \in\mathcal{G}_n}
\mathbb{E}_g[1-\psi_{n,p}]\rightarrow 0.
\]

\item[(b)] If $r_n/r^*_n\rightarrow 0$, then we have
\[
\inf_{\psi_n}\big(\mathbb{E}_0[\psi_n] + \sup_{g \in\mathcal{G}_n}
\mathbb{E}_g[1-\psi_n ]\big)\rightarrow 1,
\]
where the infimum is taken over all tests in the PPP model with intensity \eqref{EqIntensity}.
\end{itemize}
\end{thm}

\subsection{Proof of the upper bound}
Throughout the proof $C>0$ denotes a constant depending only on $R$, $\beta$ and $p$ that may change from line to line. Under the null hypothesis we have, using the Chebyshev inequality and Corollary \ref{CorPowInt},
\begin{align}\label{EqTypeIError}
\mathbb{E}_0[\psi_{n,p}]
= \mathbb{P}_0(\EstNorm \ge  r^p_n/2)
 &\le \frac{4\mathbb{E}_0[\EstNorm^2]}{r^{2p}_n}\le C\dfrac{n^{-\frac{2\beta p + 1}{\beta +1} }}{r^{2p}_n}=C\left(\dfrac{r^*_n}{r_n}\right)^{2p}
\end{align}
and by assumption the right-hand side tends to zero as $n\rightarrow\infty$. Next, consider the type-two error $\mathbb{E}_g[1-\psi_{n,p} ]$ with $g\in\mathcal{G}_n$. Let $k\in\N$ be such that $2^{k-1}r_n^p\le \|g\|_p^p< 2^kr_n^p$ and set $r_{n,k}=2^{k/p}r_n$. By the Chebyshev inequality, we have
\begin{align}
\mathbb{E}_g[1-\psi_{n,p}]= \mathbb{P}_g(\EstNorm <  r^p_n/2)
&=  \mathbb{P}_g(\|g\|_p^p -\EstNorm > \|g\|_p^p - r^p_n/2)  \nonumber\\
&\le \mathbb{P}_g(\|g\|_p^p -\EstNorm > r_{n,k}^p/4) \nonumber\\
&\le \frac{16\mathbb{E}_g[(\EstNorm-\|g\|_p^p)^2]}{r^{2p}_{n,k}}.\label{EqChevyshev}
\end{align}
Now, we may restrict ourselves to the case that
\begin{equation}\label{EqNonzeroRate}
\|g\|_{2p-2}^{2p-2}n^{-\frac{2\beta+1}{\beta+1}}\ge n^{-\frac{2\beta p+1}{\beta+1}}.
\end{equation}
Indeed, if \eqref{EqNonzeroRate} does not hold, then the maximal type-two error is also bounded by $C(r^*_n/r_n)^{2p}$, as can be seen by the same argument as in \eqref{EqTypeIError}. By \eqref{EqChevyshev}, \eqref{EqNonzeroRate} and Corollary \ref{CorPowInt}, we obtain
\begin{equation}\label{EqTypeII}
\mathbb{E}_g[1-\psi_{n,p}]\le C\Vert g \Vert^{2p-2}_{2p-2}\dfrac{n^{-\frac{2\beta+1}{\beta+1}}}{r^{2p}_{n,k}}.
\end{equation}
Let us consider the cases $1\le p \le 2$ and $p>2$ separately.
If $1< p \le 2$, then we have $\|g\|_{2p-2}\le\|g\|_p\le r_{n,k}$ by the H\"older inequality and the definition of $k$. Thus, for $1\le p \le 2$, we get
\begin{align*}
\mathbb{E}_g[1-\psi_{n,p}]
 \le  	C\frac{n^{-\frac{2\beta+1}{\beta+1}}}{r_{n,k}^{2} }\le C\left(\dfrac{r^*_n }{r_{n,k}}\right)^2 n^{- \frac{2\beta+1}{\beta+1}+
\frac{2\beta+1/p}{\beta+1}}\le C\left(\dfrac{r^*_n }{r_n}\right)^2.
\end{align*}
Taking the supremum over all $g\in\mathcal{G}_n$, the right-hand side tends to zero as $n\rightarrow \infty$. Next, consider the case $p> 2$.
Applied with $q=2p-2>p$, Lemma \ref{LemNormComp} gives
\begin{equation}\label{EqLemNormComp}
\Vert g \Vert^{2p-2}_{2p-2}\le C\|g\|_p^{2p-2}\max(1, \|g\|_p^{-1})^{ \frac{1-2/p}{\beta  +1/p}}.
\end{equation}
If $\|g\|_p>1$, then the claim follows as in the case $1\le p \le 2$.  If $\|g\|_p\le 1$, then by \eqref{EqTypeII} and \eqref{EqLemNormComp}, we have
\begin{align*}
\mathbb{E}_g[1-\psi_{n,p}]
&\le  C {r_{n,k}^{ -2-\frac{1-2/p}{\beta  +1/p}}}   n^{- \frac{2\beta+1}{\beta+1}}\\
&= C\left(\frac{r^*_n}{r_{n,k}}\right)^{ \frac{ 2\beta  +1}{\beta  +1/p}}
n^{\frac{2\beta+1}{\beta+1/p}\frac{\beta+1/2p}{\beta+1}  - \frac{2\beta+1}{\beta+1}
}\le C\left(\dfrac{r^*_n}{r_n}\right)^{ \frac{ 2\beta  +1}{\beta  +1/p}}.
\end{align*}
Again, taking the supremum over all $g\in\mathcal{G}_n$, the right-hand side tends to zero as $n\rightarrow \infty$. This completes the proof of (i). \qed

\subsection{Proof of the lower bound}\label{SecLowbound}
We set $\mathbf{P}_{1,n}(\cdot)=\int P_{g_\theta}(\cdot)\pi_n(d\theta)$ and $\mathbf{P}_{0,n}=P_0$ with $g_\theta$ and $\pi_n$ as in the proof of Theorem \ref{PropLowBoundDiff} with the choice
\begin{equation}\label{EqIdPrior}
p_1=\dots=p_m=1/\sqrt{m}.
\end{equation}
By \cite[Proposition 2.9 and Proposition 2.12]{IngsterSuslina03}, in order that Theorem \ref{TheoremTesting} (ii) holds, we have to show that as $n\rightarrow \infty$,
\begin{enumerate}
\item[(i)] $\pi_n(g_\theta\in\mathcal{G}_n)\rightarrow 1$;
\item[(ii)] $\chi^2(\mathbf{P}_{1,n},\mathbf{P}_{0,n})\rightarrow 0$.
\end{enumerate}
For (i), note that
\[
\|g_\theta\|_p=\Big(\sum_{k=1}^m\theta_k\Big)^{1/p}cRh^{\beta+1/p}\|K\|_p.
\]
By the Chebyshev inequality, we have
\begin{align*}
\pi_n\bigg(\Big(\sum_{k=1}^m\theta_k\Big)^{1/p}\le 2^{-1/p} m^{1/(2p)}\bigg)&=\pi_n\bigg(\sum_{k=1}^m(\theta_k-1/\sqrt{m})\le -\sqrt{m}/2\bigg)\\
&\le \frac{4m(1/\sqrt{m})(1-1/\sqrt{m})}{m},
\end{align*}
where the right-hand side tends to zero as $m\rightarrow\infty$. Thus (i) holds provided that we choose $m^{-1}=h$ of size
\begin{equation*}
c_1r_n^{\frac{1}{\beta+1/(2p)}}
\end{equation*}
with $c_1>0$ depending only on $R$ and $p$. Moreover, by Lemma \ref{LemChiSquare} and \eqref{EqIdPrior}, we have
\begin{equation*}
\chi^2(\mathbf{P}_{1,n},\mathbf{P}_{0,n})\le \exp\big(\exp(cRnh^{\beta+1})-1\big)-1.
\end{equation*}
Inserting the above choice of $h$, the last expression goes to zero as $n\rightarrow\infty$, since
\[
nr_n^{\frac{\beta+1}{\beta+1/(2p)}}=(r_n/r_n^*)^{\frac{\beta+1}{\beta+1/(2p)}}\rightarrow 0.
\]
This completes the proof.\qed

\section{Estimating the $L^p$-norm}\label{SecEstLp}

Finally let us consider the problem of estimating the $L^p$-norm of $g$. We define the estimator $\hat{T}$ of $\|g\|_p$ by
\[
\hat{T}=\big(\max(\EstNorm, 0)\big)^{1/p}=(\EstNorm)_+^{1/p}.
\]
Our main result of this section is as follows:
\begin{thm}\label{TheoremEstNorm}
Let $p\ge 1$ be a real number. Then we have
\[
\sup_{g\in \mathcal{C}^\beta(R)}\mathbb{E}_g[|\hat{T}-\|g\|_p|]\le Cn^{-\frac{\beta+1/(2p)}{\beta+1}}
\]
with a constant $C>0$ depending only on $R$, $\beta$ and $p$. 

On the other hand, we have
\[
\liminf_{n\rightarrow\infty}n^{\frac{\beta+1/(2p)}{\beta+1}}\inf_{\tilde T_n}\sup_{g\in \mathcal{C}^\beta(R)}\mathbb{E}_g[|\tilde{T}_n-\|g\|_p|]>0,
\]
where the infimum is taken over all estimators in the PPP Model with intensity \eqref{EqIntensity}. 
In particular, the minimax rate of estimation over $\mathcal{C}^\beta(R)$ is  $n^{-(\beta+1/(2p))/(\beta+1)}$.
\end{thm}

\begin{proof}
The lower bound follows from the lower bound in Theorem \ref{TheoremTesting}. To see this, let $r^{est}_n=\inf_{\tilde T_n}\sup_{g\in \mathcal{C}^\beta(R)}\mathbb{E}_g[|\tilde{T}_n-\|g\|_p|]$ be the minimax risk. If the lower bound in Theorem \ref{TheoremEstNorm} was false, then $r^{est}_{n_k}/r^*_{n_k}\rightarrow 0$ along a subsequence $(n_k)$. Now construct $(r_{n_k})$ such that $r_{n_k}/r_{n_k}^*\rightarrow 0$ and $r_{n_k}/r_{n_k}^{est}\rightarrow \infty$. Using \cite[Proposition 2.17]{IngsterSuslina03} and the fact that $r_{n_k}/r_{n_k}^{est}\rightarrow \infty$, we would get $\mathbb{E}_0[\psi_{n_k}] + \sup_{g \in\mathcal{G}_{n_k}}
\mathbb{E}_g[1-\psi_{n_k}]\rightarrow 0$ for suitable plug-in tests $\psi_{n_k}$ based on minimax optimal estimators, contradicting the lower bound in Theorem \ref{TheoremTesting} and the fact that $r_{n_k}/r_{n_k}^*\rightarrow 0$.

It remains to prove the upper bound.
Throughout the proof $C>0$ denotes a constant depending only on $R$, $\beta$ and $p$ that may change from line to line. Since the case $p=1$ is covered in Corollary \ref{CorPowInt}, we  restrict to the case $p>1$. By the convexity of $y\mapsto y^p$, we have (for non-negative real numbers $a\neq b$ the inequality $(b^p-a^p)/(b-a)\ge \max(a,b)^{p-1}$ holds)
\[
|\hat{T}-\|g\|_p|\le \frac{|\hat{T}^p-\|g\|_p^p|}{\|g\|_p^{p-1}}.
\]
Hence,
\begin{equation}\label{EqSmoothNorm}
\mathbb{E}_g[|\hat{T}-\|g\|_p|]\le \frac{\mathbb{E}_g[(\hat{T}^p-\|g\|_p^p)^2]^{1/2}}{\|g\|_p^{p-1}}\le \frac{\mathbb{E}_g[(\EstNorm-\|g\|_p^p)^2]^{1/2}}{\|g\|_p^{p-1}},
\end{equation}
where we also used the fact that $\hat T^p=(\EstNorm)_+$ improves $\EstNorm$ (see also Remark \ref{RemPosPart}).
On the other hand, we also have $|\hat{T}-\|g\|_p|\le |\hat{T}|+\|g\|_p$, which leads to
\begin{align}
\mathbb{E}_g[|\hat{T}-\|g\|_p|]&\le \mathbb{E}_g[\hat{T}^p]^{1/p}+\|g\|_p\nonumber\\&\le \mathbb{E}_g[|\hat{T}^p-\|g\|_p^p|]^{1/p}+2\|g\|_p \nonumber\\
& \le \mathbb{E}_g[(\EstNorm-\|g\|_p^p)^2]^{1/(2p)}+2\|g\|_p,\label{EqSingNorm}
\end{align}
where we applied the Hölder inequality and the concavity of the function $y\mapsto y^{1/p}$ (for non-negative real numbers $a\neq b$ the inequality $(a+b)^{1/p}\le a^{1/p}+b^{1/p}$ holds).

If $\norm{g}_p\le n^{-(\beta+1/(2p))/(\beta+1)}$, then by \eqref{EqSingNorm} and Corollary \ref{CorPowInt} it suffices to show
\[ \max\Big(\|g\|_{2p-2}^{2p-2}n^{-\frac{2\beta+1}{\beta+1}},n^{-\frac{2\beta p+1}{\beta+1}}\Big)^{1/(2p)}\le Cn^{-\frac{\beta+1/(2p)}{\beta+1}},
\]
which itself follows from $\|g\|_{2p-2}\le Cn^{-\beta/(\beta+1)}$. For $p\le 2$ the latter holds because of $\|g\|_{2p-2}\le\|g\|_p\le n^{-(\beta+1/(2p))/(\beta+1)}$. For $p>2$ this is implied by Lemma \ref{LemNormComp}:
\[ \|g\|_{2p-2}\le C\max(\|g\|_p,\,\|g\|_p^{(\beta+1/(2p-2))/(\beta+1/p)})\le C\|g\|_p^{\beta/(\beta+1/(2p))}\le Cn^{-\beta/(\beta+1)},
\]
using first $\|g\|_p\le 1$ and then $1/(2p-2)\ge 1/(2p)$.

In the opposite case $\norm{g}_p>n^{-(\beta+1/(2p))/(\beta+1)}$ we apply \eqref{EqSmoothNorm}, Corollary \ref{CorPowInt} and obtain the result if
\[ \max\Big(\|g\|_{2p-2}^{p-1}n^{-\frac{\beta+1/2}{\beta+1}},n^{-\frac{\beta p+1/2}{\beta+1}}\Big)\le C \|g\|_p^{p-1} n^{-(\beta+1/(2p))/(\beta+1)}.
\]
For $p\le 2$ this follows again by $\|g\|_{2p-2}\le\|g\|_p$. For $p>2$ Lemma \ref{LemNormComp} yields the bound
\[ \|g\|_{2p-2}^{p-1}\le C\|g\|_{p}^{p-1} \max(1,\|g\|_p^{(1/2-(p-1)/p)/(\beta+1/p)}\,) \le C\|g\|_p^{p-1}n^{(1/2-1/(2p))/(\beta+1)},
\]
using $((p-1)/p-1/2)(\beta+1/(2p))=(1/2-1/p)(\beta+1/(2p))<(1/2-1/(2p))(\beta+1/p)$. Inserting the bound thus gives the result also for $p>2$.
\end{proof}

\begin{remark}
For the problem of estimating $g$ in $L^\infty$-norm, Drees, Neumeyer and Selk \cite{DNS14} established the rate $(n^{-1}\log n)^{\beta/(\beta+1)}$ (in a boundary regression model). This result is then used to analyse goodness-of-fit tests for parametric classes of error distributions.
\end{remark}

\begin{remark}
Note that we can consider the minimax risk over the whole Hölder class $\mathcal{C}^\beta(R)$ in the case of estimating the norm $\|g\|_p$. In distinction to Corollary \ref{CorPowInt}, the upper bound does not depend on any $L^q$-norm of $g$.

Inspecting the proof, we see more precisely that the minimax rate is driven by functions whose $L^p$-norm is smaller than $n^{-(\beta+1/(2p))/(\beta+1)}$. For functions which have a substantially larger norm we get the rate of convergence $n^{-(\beta+1/2)/(\beta+1)}$ corresponding to a smooth functional. This is explained by the fact that the $L^p$-norm is a non-smooth functional at $g=0$.
\end{remark}
\begin{remark} There is a close connection between Theorem \ref{TheoremEstNorm} and Theorem \ref{TheoremTesting}. First, the upper bound in Theorem \ref{TheoremTesting} follows from Theorem \ref{TheoremEstNorm} by using e.g. \cite[Proposition 2.17]{IngsterSuslina03}. Second, the lower bound in Theorem \ref{TheoremEstNorm} is a consequence of the lower bound in Theorem \ref{TheoremTesting}.
\end{remark}
\appendix

\section{Proof of Lemma \ref{LemNormComp}}

Let us first show that the general case can be deduced from the special case $q=\infty$ and suppose that
\begin{equation}\label{EqNormComp}
\|f\|_\infty\le C\|f\|_p\max\big(1,R/\|f\|_p\big)^{\frac{1/p}{\beta +1/p}}
\end{equation}
holds. Clearly, we have
\begin{equation}\label{EqStart}
\|f\|_q^q\le \|f\|_\infty^{q-p}\|f\|_p^p.
\end{equation}
Now, if $\|f\|_p>R$, then \eqref{EqNormComp} and \eqref{EqStart} give $\|f\|_{q}\le C^{1-p/q}\|f\|_{p}$. On the other hand, if $\|f\|_p\le R$, then \eqref{EqNormComp} and \eqref{EqStart} give
\begin{equation*}
\|f\|_{q}^q\le C^{q-p}\|f\|_p^q(R/\|f\|_p)^{\frac{(q-p)/p}{\beta +1/p}}
\end{equation*}
and thus
\begin{equation*}
\|f\|_{q}\le C^{1-p/q}\|f\|_p(R/\|f\|_p)^{\frac{1/p-1/q}{\beta +1/p}}.
\end{equation*}
It remains to prove \eqref{EqNormComp}. Using the definition of $\mathcal{C}^\beta(R)$, we get
\begin{equation*}
\|f\|_p^p=\int_0^1|f(x)|^p\,dx\ge \int_0^{\min(1,(\|f\|_\infty/R)^{1/\beta})}(\|f\|_\infty-Rx^\beta)^p\,dx.
\end{equation*}
Setting $a=\|f\|_\infty$ and $b=(\|f\|_\infty/R)^{1/\beta}$, we obtain
\begin{align*}
\int_0^1|f(x)|^p\,dx\ge \int_0^{\min(1,b)}(a-a(x/b)^\beta)^p\,dx
&=a^p\int_0^{\min(1,b)}(1-(x/b)^\beta)^p\,dx\\
&\ge a^p\min(1,b)\int_0^1(1-y^\beta)^p\,dy,
\end{align*}
where we make the substitution $x=by$ if $b\le 1$ and use the inequality $1-(x/b)^\beta\ge 1-x^\beta$ if $b>1$. Thus we have proven
\begin{align*}
\|f\|_p\ge \|f\|_\infty\min\big(1,\|f\|_\infty/R \big)^{\frac{1}{\beta p}}\|1-y^\beta\|_p,
\end{align*}
which gives \eqref{EqNormComp}.

\bibliographystyle{plain}
\bibliography{lit}

\end{document}